\documentclass[12pt,a4paper]{amsart}

\usepackage{amsfonts,amsmath,amssymb,indentfirst,mathrsfs,amscd}

\numberwithin{equation}{section}

\title{Brauer group of moduli spaces of pairs}
\date{6 June 2010. Revised: 15 December 2010}

\author[I. Biswas]{Indranil Biswas}
\address{School of Mathematics, Tata Institute of Fundamental
Research, Homi Bhabha Road, Bombay 400005, India}
\email{indranil@math.tifr.res.in}

\author[M. Logares]{Marina Logares}
\address{Instituto de Ciencias Matem\'aticas (CSIC-UAM-UC3M-UCM),
Serrano 113bis, 28006 Madrid, Spain}
\email{marina.logares@icmat.es}

\author[V. Mu\~{n}oz]{Vicente Mu\~{n}oz}
\address{Facultad de Matem\'aticas, Universidad Complutense de Madrid,
Plaza Ciencias 3, 28040 Madrid Spain}
\email{vicente.munoz@mat.ucm.es}



\DeclareMathOperator{\rk}{rk\,}          
\DeclareMathOperator{\Ext}{Ext}       
\DeclareMathOperator{\Mat}{Mat}
\DeclareMathOperator{\Hom}{Hom}
\DeclareMathOperator{\Br}{Br}
\DeclareMathOperator{\Jac}{Jac}
\DeclareMathOperator{\codim}{codim}

\newcommand{\hra}{\hookrightarrow}

\newcommand{\cF}{\mathcal{F}}

\newcommand{\cO}{\mathcal{O}}
\newcommand{\cS}{\mathcal{S}}
\newcommand{\cU}{\mathcal{U}}

\newcommand{\PP}{\mathbb{P}} 
\newcommand{\RR}{\mathbb{R}} 
\newcommand{\GG}{\mathbb{G}} 
\newcommand{\ZZ}{\mathbb{Z}} 

\newcommand{\lM}{\mathfrak{M}} 

\begin{document}

\newtheorem{thm}{Theorem}[section]
\newtheorem{prop}[thm]{Proposition}
\newtheorem{lem}[thm]{Lemma}
\newtheorem{cor}[thm]{Corollary}
\newtheorem{question}[thm]{Question}

\theoremstyle{definition}
\newtheorem{defn}[thm]{Definition}

\theoremstyle{remark}
\newtheorem{rmk}[thm]{Remark}

\theoremstyle{remark}
\newtheorem*{prf}{Proof}

\begin{abstract}
We show that the Brauer group of the moduli space of stable pairs with
fixed determinant over a curve is zero.
\end{abstract}

\maketitle
\section{Introduction}

Let $X$ be a smooth projective curve of genus $g\, \geq\, 2$ over the complex numbers.
A \textit{holomorphic pair} (also called a \textit{Bradlow pair}) is an object of the form
$(E,\phi)$, where $E$ is a holomorphic vector bundle over $X$,
and $\phi$ is a nonzero holomorphic section of
$E$. The concept of stability for pairs depends on a parameter $\tau\in \RR$.
Moduli spaces of $\tau$-stable pairs of fixed rank and degree
were first constructed using gauge theoretic
methods in \cite{bradlow-daskalopoulos:1991}, and subsequently
using Geometric Invariant Theory in \cite{bertram:1994}.
Since then these moduli spaces have been extensively studied.

Fix an integer $r\, \geq\, 2$ and a holomorphic line bundle $\Lambda$
over $X$. Let $d=\deg(\Lambda)$. Let
$\lM_{\tau}(r,\Lambda)$ be the
moduli space of stable pairs $(E,\phi)$ such that $\rk (E)\,=\,r$
and $\det (E)\, =\, \bigwedge^r E\,=\, \Lambda$. This
is a smooth quasi-projective variety; it is empty if $d \leq 0$. Therefore, $H^{2}_{\mathaccent 19 e
t}(\lM_{\tau}(r,\Lambda),\GG_m )$ is torsion, and it coincides with
the Brauer group of $\lM_{\tau}(r,\Lambda)$, defined by the equivalence classes of
Azumaya algebras over $\lM_{\tau}(r,\Lambda)$. Let
$\Br(\lM_{\tau}(r,\Lambda))$ denote the Brauer group of
$\lM_{\tau}(r,\Lambda)$.

We prove the following (see Theorem \ref{theorem1} and
Corollary \ref{cor:extra}):

\begin{thm} \label{main1}
 Assume that $(r,g,d)\neq (3,2,2)$. Then $\Br(\lM_{\tau}(r,\Lambda))\,=\, 0$.
\end{thm}

Let $M(r,\Lambda)$ be the moduli space of stable vector bundles
over $X$ of rank $r$ and determinant $\Lambda$. There is a unique
universal projective bundle over $X\times M(r,\Lambda)$. Restricting
this projective bundle to $\{x\}\times M(r,\Lambda)$, where $x$ is
a fixed point of $X$, we get a projective bundle ${\mathbb P}_x$ over
$M(r,\Lambda)$. We give a new proof of the following known result (see
Corollary \ref{cor1}).

\begin{cor} \label{main2}
 Assume $(r,g,d)\neq (2,2,\text{even})$.
 The Brauer group of $M(r,\Lambda)$ is generated by the Brauer class of ${\mathbb P}_x$.
\end{cor}

This was first proved in \cite{balaji-biswas-gabber-nagaraj:2007}.
We show that it follows as an application of Theorem \ref{main1}.

For convenience, we work over the complex numbers. However, the results
are still valid for any algebraically closed field of characteristic zero.

\medskip

\noindent \textbf{Acknowledgements.} We are grateful to
Norbert Hoffmann and Peter Newstead for helpful comments; specially,
Peter Newstead pointed out a mistake in a previous version of
Proposition \ref{prop:Um-UM}. We thank the referee for a careful reading and
useful comments. The second and third authors are grateful to the Tata Institute of Fundamental
Research (Mumbai), where this work was carried out, for its hospitality.
Second author supported by (Spanish MICINN) research project MTM2007-67623 and i-MATH. First and third author partially
supported by (Spanish MCINN) research project MTM2007-63582.

\section{Moduli spaces of pairs}

We collect here some known results about the moduli spaces of pairs, taken mainly from
\cite{bradlow-daskalopoulos:1991}, \cite{bradlow-garciaprada:1996},
\cite{munoz-ortega-vazquezgallo:2007}, \cite{munoz:2009} and \cite{thaddeus:1994}.

Let $X$ be a smooth projective curve defined over the field of complex
numbers, of genus $g\, \geq \, 2$.
A \emph{holomorphic pair} $(E,\phi)$ over $X$ consists of a holomorphic bundle on $X$ and a nonzero
holomorphic section
$\phi\in H^{0}(E)$. Let $\mu(E):=\deg(E)/\rk(E)$ be the slope of $E$. There is a stability concept for a pair
depending on a parameter $\tau\in \RR$.
A holomorphic pair $(E,\phi)$
is \emph{$\tau$-stable} whenever the following conditions are satisfied:
\begin{itemize}
  \item for any subbundle $E'\subset E$, we have $\mu(E')<\tau$,
  \item for any subbundle $E'\subset E$ such that $\phi\in H^0 (E')$, we have $\mu (E/E')>\tau$.
\end{itemize}
The concept of \emph{$\tau$-semistability} is defined by replacing the above strict inequalities
by the weaker inequalities ``$\le$'' and ``$\ge$''.
A \emph{critical value} of the parameter $\tau=\tau_c$ is one
for which there are strictly $\tau$-semistable pairs.
There are only finitely many critical values.

Fix an integer $r\, \geq\,2$ and
a holomorphic line bundle $\Lambda$ over $X$. Let $d$
be the degree of $\Lambda$.
We denote by $\lM_{\tau}(r,\Lambda)$ (respectively,
$\overline{\lM}_\tau(r,\Lambda)$)
the moduli space of $\tau$-stable (respectively, $\tau$-semistable)
pairs $(E,\phi)$ of rank $\rk(E)=r$ and determinant $\det(E)=\Lambda$.
The moduli space $\overline{\lM}_{\tau}(r,\Lambda)$ is
a normal projective variety, and $\lM_{\tau}(r,\Lambda)$ is
a smooth quasi-projective variety contained in the smooth locus
of $\overline{\lM}_{\tau}(r,\Lambda)$ (cf. \cite[Theorem 3.2]{munoz:2009}).
Moreover, $\dim \lM_\tau (r,\Lambda)= d+ (r^2-r-1)(g-1)-1$.

For non-critical values of the parameter, there
are no strictly $\tau$-semistable pairs, so
$\lM_{\tau}(r,\Lambda)\,=\,\overline{\lM}_{\tau}(r,\Lambda)$ and it is a smooth
projective variety. For a critical value $\tau_c$, the variety
$\overline{\lM}_{\tau_c}(r,\Lambda)$ is in general singular.

Denote  $\tau_m :=\frac{d}{r}$ and  $\tau_M:=\frac{d}{r-1}$. The moduli space
$\lM_\tau(r,\Lambda)$ is empty for $\tau \not\in (\tau_m,\tau_M)$. In particular, this forces $d>0$ for
$\tau$-stable pairs. Denote by
$\tau_1<\tau_2<\ldots <\tau_L$ the collection of all critical values in $(\tau_m,\tau_M)$.
Then the moduli spaces $\lM_\tau(r,\Lambda)$ are isomorphic for all values $\tau$ in
any interval $(\tau_i,\tau_{i+1})$, $i=0,\ldots, L$; here $\tau_0=\tau_m$ and
$\tau_{L+1}=\tau_M$.

However, the moduli space changes when we cross a critical value. Let $\tau_c$
be a critical value (note that for us, a critical value $\tau_c\neq \tau_m,\tau_M$).
Denote $\tau_{c}^{+}:=\tau_{c}+\epsilon$ and
$\tau_{c}^{-}:=\tau_{c}-\epsilon$ for $\epsilon >0$ small enough such that $(\tau_{c}^{-},\tau_{c}^{+})$ does not
contain any critical value other than $\tau_{c}$. We define the \emph{flip
loci} $\cS_{\tau^{\pm}_{c}}$ as the subschemes:
\begin{itemize}
  \item $\cS_{\tau^{+}_{c}}=\{(E,\phi)\in \lM_{\tau_c^{+}} (r,\Lambda)\, |
   \, (E,\phi)\; \text{is $\tau_{c}^{-}$-unstable}\}$,
  \item  $\cS_{\tau^{-}_{c}}=\{(E,\phi)\in \lM_{\tau_c^{-}}(r,\Lambda)\, | \,
    (E,\phi)\; \text{is $\tau_{c}^{+}$-unstable}\}$.
\end{itemize}
When crossing $\tau_{c}$, the variety $\lM_{\tau}(r,\Lambda)$ undergoes a birational transformation:
 $$
 \lM_{\tau^{-}_{c}}(r,\Lambda)\setminus
 \cS_{\tau_{c}^{-}}\,=\, \lM_{\tau_{c}}(r,\Lambda)=\lM_{\tau^{+}_{c}}(r,\Lambda)\setminus
 \cS_{\tau_{c}^{+}}\, .
 $$

\begin{prop}[{\cite[Proposition 5.1]{munoz:2008}}] \label{prop:codim}
  Suppose $r\ge2$, and let $\tau_c$ be a critical value with $\tau_m<\tau_c<\tau_M$. Then
  \begin{itemize}
    \item $\codim \cS_{\tau_{c}^{+}}\ge 3$ except in the case $r=2$, $g=2$, $d$ odd and $\tau_{c}=\tau_{m}+\frac{1}{2}$ (in which case $\codim \cS_{\tau_{c}^{+}}=2$),
    \item $\codim \cS_{\tau_{c}^{-}}\ge 2$ except in the case $r=2$ and $\tau_{c}=\tau_{M}-1$ (in which case $\codim \cS_{\tau_{c}^{-}}=1$). Moreover we have that $\codim \cS_{\tau_{c}^{-}}=2$ only for $\tau_{c}=\tau_{M}-2$.
  \end{itemize}
\end{prop}

The codimension of the flip loci is then always positive, hence we have the following corollary:

\begin{cor} \label{cor:birationality}
  The moduli spaces $\lM_{\tau}(r,\Lambda)$,
$\tau\in (\tau_{m},\tau_{M})$, are birational.
\end{cor}

The moduli spaces for the extreme values of the parameter $\tau_{m}^{+}$ and
$\tau_{M}^{-}$ are known explicitly.
Let $M(r,\Lambda)$ be the moduli space of \textit{stable} vector bundles or rank $r$ and
fixed determinant $\Lambda$. Define
\begin{equation}\label{eq.ll}
  \cU_m(r,\Lambda)=\{ (E,\phi) \in \lM_{\tau_m^+}(r,\Lambda) \, | \, \text{$E$ is a stable vector bundle}\}\, ,
\end{equation}
and denote
  $$
  \cS_{\tau_m^+}:= \lM_{\tau_{m}^{+}}(r,\Lambda)\setminus \cU_m(r,\Lambda) \,
  $$
(not to be confused with the definition above for $\cS_{\tau_c^\pm}$, which refers only to critical values
$\tau_c\neq \tau_m,\tau_M$). 
Then there is a map
 \begin{equation}\label{eq:ftaum}
 \pi_1:\cU_m  (r,\Lambda)\longrightarrow M(r,\Lambda),\qquad (E,\phi)\mapsto E\, ,
 \end{equation}
whose fiber over $E$ is the projective space $\PP(H^{0}(E))$.
When $d \geq r(2g-2)$, we have
that $H^1(E)=0$ for any stable bundle, and hence (\ref{eq:ftaum}) is a projective bundle
(cf.\ \cite[Proposition 4.10]{munoz-ortega-vazquezgallo:2007}).

\medskip

Regarding the rightmost moduli space $\lM_{\tau_{M}^{-}}(r,\Lambda)$, we have that
any $\tau_M^-$-stable pair $(E,\phi)$ sits in an exact sequence
 $$
 0\longrightarrow \cO \stackrel{\phi}{\longrightarrow} E\longrightarrow F\longrightarrow 0\, ,
 $$
where $F$ is a semistable bundle of rank $r-1$ and $\det (F)=\Lambda$. Let
 $$
 \cU_M(r,\Lambda)=\{ (E,\phi) \in \lM_{\tau_M^-}(r,\Lambda) \, | \, \text{$F$
is a stable vector bundle}\}\, ,
  $$
and denote
  $$
  \cS_{\tau_M^-}:= \lM_{\tau_M^-} (r,\Lambda)\setminus \cU_M (r,\Lambda)\, .
  $$
Then there is a map
\begin{equation}\label{eq:ftauM}
 \pi_2: \cU_M(r,\Lambda) \longrightarrow  M(r-1,\Lambda), \qquad (E,\phi) \mapsto E/\phi(\cO) \, ,
\end{equation}
whose fiber over $F \in M(r-1,\Lambda)$ is the
projective spaces $\PP(H^1(F^\ast))$
(cf. \cite{bradlow-garciaprada:2002}). Note
that $H^0(F^\ast)=0$ since $d>0$.
So (\ref{eq:ftauM}) is always a projective bundle.

In the particular case of rank $r=2$,
the rightmost moduli space is
  \begin{equation}\label{eq:r2tauM}
  \lM_{\tau_{M}^{-}}(2,\Lambda)=\PP(H^{1}(\Lambda^{-1}))\, ,
  \end{equation}
since $M(1,\Lambda)=\{\Lambda\}$.
In particular, Corollary \ref{cor:birationality} shows that all
$\lM_{\tau}(2,\Lambda)$ are rational quasi-projective varieties.

We have the following:
\begin{lem}[{\cite[Lemma 5.3]{munoz:2009}}] \label{lem:codim-semistable}
Let $S$ be a bounded family of isomorphism classes of strictly semistable
bundles of rank $r$ and determinant $\Lambda$. Then $\dim M (r, \Lambda)
 - \dim S \geq (r - 1)(g - 1)$.
\end{lem}

\begin{prop} \label{prop:Um-UM}
 The following two statements hold:
   \begin{itemize}
    \item Suppose $d>r(2g-2)$. Then $\codim \cS_{\tau_{m}^{+}}\ge 2$ except in the case $r=2$,
$g=2$, $d$ even
     (in which case $\codim \cS_{\tau_{m}^{+}}=1$).
    \item Suppose $r\geq 3$. Then
    $\codim \cS_{\tau_{M}^{-}}\ge 2$ except in the case $r=3$, $g=2$,
$d$ even
     (in which case the $\codim \cS_{\tau_{M}^{-}}=1$).
  \end{itemize}
\end{prop}

\begin{proof}
Let $(E,\phi)\in \lM_{\tau_m^+}(r,\Lambda)$, then $E$ is a semistable
bundle. As $d>r(2g-2)$, $H^1(E)=H^0(E^*\otimes K_X)^*=0$, since
$E^*\otimes K_X$ is semistable and has negative degree. Therefore,
$\dim H^0(E)$ is constant. Let $\cF$ be the family of strictly
semistable bundles $E$ such that there exists some $\phi$ with
$(E,\phi)
\in \cS_{\tau_{m}^{+}}$. Then $\codim \cS_{\tau_m^+} =
\dim \lM_{\tau_m^+}(r,\Lambda) -
\dim \cS_{\tau_m^+} \geq \dim M(r,\Lambda) - \dim \cF \ge (r-1)(g-1)$
(by
Lemma \ref{lem:codim-semistable}). The first statement follows.

As the dimension $\dim H^1(F^\ast)$ is constant, the codimension of $\cS_{\tau_M^-}$ in
$\lM_{\tau_M^-}(r,\Lambda)$ is at least
the codimension of a locus of semistable bundles. Applying Lemma
\ref{lem:codim-semistable} to $M (r-1, \Lambda)$ we have
$\codim \cS_{\tau_M^-} \geq (r-2)(g-1)$. The second item follows.
\end{proof}

\section{Brauer group}

The Brauer group of a scheme $Z$ is defined as the equivalence classes of Azumaya algebras on $Z$,
that is, coherent locally free sheaves with algebra structure such that,
locally on the \'etale topology of $Z$, are
isomorphic to a matrix algebra $\Mat(\cO_Z)$. If $Z$ is a smooth quasi-projective variety, then
the Brauer group $\Br(Z)$ coincides with $H^{2}_{\mathaccent 19 e
t}(Z)$, and $H^{2}_{\mathaccent 19 e t}(Z)$ is a torsion group.

\begin{thm}\cite[VI.5 (Purity)]{milne:1980}\label{thm:purity}
Let $Z$ be a smooth complex variety and $U\subset Z$ be a Zariski open subset
whose complement has codimension at least $2$. Then $\Br(Z)\,=\,\Br(U)$.
\end{thm}

On the moduli space of stable vector bundles $M(r,\Lambda)$,
there are three natural projective bundles. We will describe them.

We first note that there
is a unique universal projective bundle over $X\times M(r,\Lambda)$.
Fix a point $x\, \in\, X$. Restricting the universal projective bundle to
$\{x\} \times M(r,\Lambda)$ we get a projective bundle
 \begin{equation}\label{eqn:px}
 {\mathbb P}_x\longrightarrow M(r,\Lambda) \, .
 \end{equation}

Secondly, if $d \geq r(2g-2)$, then we have the projective bundle
 \begin{equation}\label{eqn:p0}
 {\mathcal P}_0\longrightarrow M(r,\Lambda)\, ,
 \end{equation}
whose fiber over any $E\in M(r,\Lambda)$ is the projective space ${\mathbb
P}(H^0(E))$; note that we have $H^1(E)\,=\, 0$ because $d \geq r(2g-2)$.

Finally, assuming $d>0$, let
 \begin{equation}\label{eqn:p1}
 {\mathcal P}_1\longrightarrow M(r,\Lambda)
 \end{equation}
be the projective bundle whose fiber over any $E\in
M(r,\Lambda)$ is the projective space ${\mathbb P}(H^1(E^*))$.

\begin{prop}\label{pr-br}
The Brauer class ${\rm cl}({\mathbb P}_x)\, \in\, \Br(M(r,\Lambda))$
is independent of $x\, \in \, X$. Moreover,
 $$
 {\rm cl}({\mathbb P}_x) \, =\,
 {\rm cl}({\mathcal P}_0)\,=\, - {\rm cl}({\mathcal P}_1)\, ,
 $$
when they are defined.
\end{prop}

\begin{proof}
The moduli space $M(r,\Lambda)$ is constructed as a Geometric
Invariant Theoretic quotient of a Quot scheme $\mathcal Q$ by the action
of a linear group $\text{GL}_N(\mathbb C)$ (see \cite{Ne}). The isotropy
subgroup for a stable point of $\mathcal Q$ is the
center ${\mathbb C}^*\, \subset\, \text{GL}_N(\mathbb C)$. There
is a universal vector bundle
 $$
 {\mathcal E} \longrightarrow X\times \mathcal Q\, .
 $$
Let $Z(\text{GL}_N(\mathbb C))$ be the center of
$\text{GL}_N(\mathbb C)$.
The action of the subgroup $Z(\text{GL}_N(\mathbb C))$ on
$\mathcal Q$ is trivial. Therefore, $Z(\text{GL}_N(\mathbb C))$
acts on each fiber of ${\mathcal E}$.
Identify $Z(\text{GL}_N(\mathbb C))$ with
${\mathbb C}^*$ by sending any $\lambda\, \in\, {\mathbb C}^*$
to $\lambda\cdot \text{Id}$. We note that
$\lambda\, \in\, Z(\text{GL}_N(\mathbb C))$ acts
on ${\mathcal E}$ as multiplication by $\lambda$.

Let ${\mathcal Q}^s\, \subset\, \mathcal Q$ be the stable locus.
The restriction of $\mathcal E$ to $X\times{\mathcal Q}^s$ will be
denoted by ${\mathcal E}^s$. Let
 $$
 {\mathcal E}_x\, :=\, {\mathcal E}^s\vert_{\{x\}\times{\mathcal Q}^s}
 \longrightarrow{\mathcal Q}^s
 $$
be the restriction. Let $p_2\, :\,  X\times {\mathcal
Q}^s\longrightarrow{\mathcal Q}^s$ be the
natural projection. Define the vector bundles
 $$
 {\mathcal E}_0\, :=\, p_{2*}\, {\mathcal E}^s ~\,~\,\text{~and~}
 ~\,~\,{\mathcal E}_1\, :=\, R^1p_{2*}(({\mathcal E}^s)^*)\, .
 $$

We noted that any $\lambda\,\in\, {\mathbb C}^*\,=\,
Z(\text{GL}_N(\mathbb C))$ acts on ${\mathcal E}_x$ as
multiplication by $\lambda$. Therefore, $\lambda$ acts
on $({\mathcal E}^s)^*$ as multiplication by $1/\lambda$.
Hence $\lambda$ acts on ${\mathcal E}_1$ as multiplication by
$1/\lambda$. Consequently, the action of $Z(\text{GL}_N(\mathbb C))$
on ${\mathcal E}_x\otimes {\mathcal E}_1$ is trivial. Hence
${\mathcal E}_x\otimes {\mathcal E}_1$ descends to a vector
bundle over the quotient $M(r,\Lambda)$ of ${\mathcal Q}^s$.
Therefore,
 $$
 {\rm cl}({\mathbb P}_x)\,=\, - {\rm cl}({\mathcal P}_1)\, .
 $$

Any $\lambda\, \in\, {\mathbb C}^*\,=\, Z(\text{GL}_N(\mathbb C))$
acts on ${\mathcal E}_0$ as multiplication by $\lambda$. Indeed,
this follows immediately from the fact that $\lambda$ acts
as multiplication by $\lambda$ on ${\mathcal E}^s$.
As noted earlier,  $\lambda$ acts on ${\mathcal E}_1$ as multiplication
by $1/\lambda$. Hence the action of $Z(\text{GL}_N(\mathbb C))$ on
${\mathcal E}_0\otimes {\mathcal E}_1$ is trivial.
Thus ${\mathcal E}_0\otimes {\mathcal E}_1$ descends to
$M(r,\Lambda)$, implying
 $$
 {\rm cl}({\mathcal P}_0)\,=\, - {\rm cl}({\mathcal P}_1)\, .
 $$

Finally, note that it follows that ${\rm cl}({\mathbb P}_x)$ is
independent of $x\,  \in\,  X$ for $d>0$. For $d\leq 0$, ${\mathcal P}_0$
and ${\mathcal P}_1$ are not defined. In this case, we take a line
bundle $\mu$ or large degree, and use the isomorphism
$M(r,\Lambda\otimes \mu^r) \cong M(r,\Lambda)$. For any
pair $x,x'\, \in\, X$, since ${\rm cl}({\mathbb
P}_x)=
{\rm cl}({\mathbb P}_{x'})$ in $\Br(M(r,\Lambda\otimes \mu^r))$,
the same holds for $\Br(M(r,\Lambda))$.
\end{proof}

\begin{thm}\label{theorem1}
Assume that $d > r(2g-2)$. Then
for the moduli space $\lM_{\tau}(r,\Lambda)$ of stable pairs, we have that
  $$
  \Br(\lM_{\tau}(r,\Lambda))\,=\,0 \, .
  $$
\end{thm}

\begin{proof}
We will first prove it for $r\, =\, 2$. Recall from (\ref{eq:r2tauM})
that $\lM_{\tau_{M}^{-}}(2,\Lambda)$ is a projective space, hence
  $$
  \Br(\lM_{\tau_{M}^{-}}(2,\Lambda))=0\, .
  $$
Moreover, all $\lM_{\tau}(2,\Lambda)$ are rational varieties. Thus
  $$
  \Br(\lM_{\tau}(2,\Lambda))=0
  $$
for non-critical values $\tau\in (\tau_m,\tau_M)$,
since the Brauer group of a smooth rational projective variety
is zero \cite[p. 77, Proposition 1]{AM}. For a critical value $\tau_c$, we have
 $$
 \lM_{\tau_{c}}(2,\Lambda)= \lM_{\tau_{c}^+}(2,\Lambda) \setminus \cS_{\tau_{c}^+}\, .
 $$
By Proposition \ref{prop:codim}, $\codim \cS_{\tau_{c}^{+}} \geq 2$, so the
Purity Theorem implies that
  $$
  \Br(\lM_{\tau_c}(2,\Lambda))=0\, .
  $$

Now we assume that $r\,\ge\, 3$. From Proposition \ref{prop:codim} and
Theorem \ref{thm:purity} it
follows that the Brauer group $\Br(\lM_{\tau}(r,\Lambda))$ does not
depend on the value of the parameter $\tau$ (for fixed $r$ and
$\Lambda$).

As 
$d \geq r(2g-2)$, we have a projective
bundle $$\pi_1 : \cU_m (r,\Lambda) \longrightarrow M(r,\Lambda)$$
(see \eqref{eq:ftaum}). Note that this projective bundle
coincides with the projective bundle ${\mathcal P}_0$ in
(\ref{eqn:p0}). The projective bundle $\pi_1$ gives an exact sequence
 \begin{equation}\label{e0}
 {\mathbb Z}\cdot \text{cl}({\mathcal P}_0)\,\longrightarrow\,
 \Br(M(r,\Lambda))\,\longrightarrow\,
 \Br(\cU_m (r,\Lambda)) \,\longrightarrow\, 0
 \end{equation}
(see \cite[p.\ 193]{Ga}). As $d>r(2g-2)$, Proposition \ref{prop:Um-UM} and the Purity
Theorem give
  $$
  \Br(\cU_m(r,\Lambda)) \, = \, \Br(\lM_{\tau_{m}^{+}} (r,\Lambda))\, ,
  $$
so we have
 \begin{equation}\label{e1}
 {\mathbb Z}\cdot \text{cl}({\mathcal P}_0)\,\longrightarrow\,
 \Br(M(r,\Lambda))\,\longrightarrow\,
 \Br(\lM_{\tau_{m}^{+}}(r,\Lambda)) \,\longrightarrow\, 0\, .
 \end{equation}

We will show that the theorem follows from \eqref{e1} if we use
\cite{balaji-biswas-gabber-nagaraj:2007}. From Proposition
\ref{pr-br} we know that $\text{cl}({\mathcal P}_0)\,=\,
{\rm cl}({\mathbb P}_x)$, and from
\cite[Proposition 1.2(a)]{balaji-biswas-gabber-nagaraj:2007} we know
that ${\rm cl}({\mathbb P}_x)$ generates $\Br(M(r,\Lambda))$.
Therefore, from \eqref{e1} it follows that
 $$
 \Br(\lM_{\tau_{m}^{+}}(r,\Lambda))\,=\,0\, .
 $$
Since $\Br(\lM_{\tau}(r,\Lambda))$ is independent of $\tau$, this
completes the proof using \cite{balaji-biswas-gabber-nagaraj:2007}.
But we shall give a different proof without using
\cite{balaji-biswas-gabber-nagaraj:2007}, because we want
to show that the above mentioned result of
\cite{balaji-biswas-gabber-nagaraj:2007} can be deduced from
our Theorem \ref{theorem1} (see Corollary \ref{cor1}).

Consider the projective bundle $\pi_2: \cU_M (r-1,\Lambda)\,
\longrightarrow\,  M(r-1,\Lambda)$ from
\eqref{eq:ftauM}. Note that this projective bundle coincides with
the projective bundle ${\mathcal P}_1$ in (\ref{eqn:p1}) for rank $r-1$.
The projective bundle $\pi_2$ gives an exact sequence
  \begin{equation}\label{e2}
  {\mathbb Z}\cdot \text{cl}({\mathcal P}_1)\,\longrightarrow\,
  \Br(M(r-1,\Lambda))\,\longrightarrow\,
  \Br(\cU_M(r,\Lambda))=\Br(\lM_{\tau_{M}^{-}}(r,\Lambda)) \,\longrightarrow\, 0\, ,
  \end{equation}
using Proposition \ref{prop:Um-UM}, with the exception of the case $(r,g,d)=(3,2, \text{even})$.
Let us leave this ``bad'' case aside for the moment.

Let
  \begin{equation}\label{e3}
 {\mathbb Z}\cdot \text{cl}({\mathcal P}_0)\,\longrightarrow\,
 \Br(M(r-1,\Lambda))\,\longrightarrow\,
 \Br(\cU_m(r-1,\Lambda))=\Br(\lM_{\tau_{m}^{+}}(r-1,\Lambda)) \,\longrightarrow\, 0
 \end{equation}
be the exact sequence obtained by replacing $r$ with $r-1$ in
\eqref{e1}; the last equality holds as $(r-1,g,d)\neq (2,2,\text{even})$, by
Proposition \ref{prop:Um-UM}.

Since $\text{cl}({\mathcal P}_1)\,=\, -\text{cl}({\mathcal
P}_0)$ (see Proposition \ref{pr-br}), comparing \eqref{e2} and
\eqref{e3} we conclude that the two quotients of
$\Br(M(r-1,\Lambda))$, namely
 $$
 \Br(\lM_{\tau_{M}^{-}}(r,\Lambda))\,~\,~ \text{~and~}\,~\,~
 \Br(\lM_{\tau_{m}^{+}}(r-1,\Lambda))\, ,
 $$
coincide. In particular,
$\Br(\lM_{\tau_{M}^{-}}(r,\Lambda))$ is isomorphic to
$\Br(\lM_{\tau_{m}^{+}}(r-1,\Lambda))$. Therefore, using induction,
the group $\Br(\lM_{\tau_{M}^{-}}(r,\Lambda))$ is isomorphic
to $\Br(\lM_{\tau_{m}^{+}}(2,\Lambda))$. We have already shown
that $\Br(\lM_{\tau_{m}^{+}}(2,\Lambda))\,=\,0$. Hence the
proof of the theorem is complete for $d > r(2g-2)$ and  $(r,g,d) \neq (3,2, \text{even})$.

Let us now investigate the missing case of
$(r,g,d)=(3,2, 2k)$. Take a line bundle $\nu$ of degree $1$. Using (\ref{e0}) twice, we have
 $$
   \begin{array}{ccccccc}
  {\mathbb Z}\cdot \text{cl}({\mathcal P}_0) &\longrightarrow &  \Br(M(3,\Lambda))
  &\longrightarrow & \Br(\cU_m (3,\Lambda)) &\longrightarrow& 0  \\
   \downarrow \cong && || \\
  {\mathbb Z}\cdot \text{cl}({\mathcal P}_0) &\longrightarrow &
\Br(M(3,\Lambda\otimes \nu^3))
  &\longrightarrow & \Br(\cU_m (3,\Lambda\otimes\nu^3)) &\longrightarrow& 0  \\
  \end{array}
  $$
The second vertical map is induced by the isomorphism
$M(3,\Lambda)\longrightarrow M(3,\Lambda\otimes \nu^3)$
defined by $E\mapsto E\otimes \nu$, hence it
is an isomorphism. This isomorphism preserves the class
$\text{cl}({\mathbb P}_x)$, and hence the class $\text{cl}({\mathcal P}_0)$, by Proposition
\ref{pr-br}. Therefore, $\Br(\cU_m(3,\Lambda))=\Br(\cU_m(3,\Lambda\otimes \nu^3))$.
But $\deg(\Lambda \otimes \nu^3)$ is odd, hence
 $$
 \Br(\cU_m(3,\Lambda))=\Br(\cU_m(\Lambda\otimes \nu^3))=0\, .
  $$
By the Purity Theorem, $\Br(\lM_\tau(3,\Lambda))=0$ for any $\tau$.
\end{proof}

Note that the proof of Theorem \ref{theorem1} works in the following
cases:
\begin{itemize}
\item{} $r=2$, any $d$\, ;
\item{} $(r,g,d)\neq (3,2,\text{even})$, $d > (r-1)(2g-2)$\, ; and
\item{} $r=3$, $g=2$, $d > 6$\, .
\end{itemize}

Before proceeding to remove the assumption $d > r(2g-2)$ in Theorem
\ref{theorem1},
we want to show that Theorem \ref{theorem1} implies
Proposition 1.2(a) of \cite{balaji-biswas-gabber-nagaraj:2007}.

\begin{cor}\label{cor1}
Suppose that $(r,g,d)\neq (2,2,\text{even})$.
The Brauer group $\Br(M(r,\Lambda))$ is generated by the
Brauer class ${\rm cl}({\mathbb P}_x)\, \in\, \Br(M(r,\Lambda))$
in \eqref{eqn:px}.
\end{cor}

\begin{proof}
Without loss of generality we can assume that $d$ is large (since we have an isomorphism
$M(r,\Lambda) \,\stackrel{\sim}{\longrightarrow}\, M(r,\Lambda \otimes \mu^r)$,
$E\mapsto E\otimes \mu$, where $\mu$ is a line bundle.

First, we have $\Br(\cU_m(r,\Lambda))= \Br(\lM_{\tau_{m}^{+}}(r,\Lambda))$ by
the Purity Theorem and Proposition \ref{prop:Um-UM}. Second, $\Br(\lM_{\tau_{m}^{+}}(r,\Lambda))=0$
by Theorem \ref{theorem1}, so $\Br(\cU_m(r,\Lambda))= 0$. Finally, we use the
exact sequence in
\eqref{e0} to see that $\text{cl}({\mathcal P}_0)$
generates $\Br(M(r,\Lambda))$. Now from Proposition \ref{pr-br}
it follows that ${\rm cl}({\mathbb P}_x)$ generates $\Br(M(r,\Lambda))$.
\end{proof}

\begin{cor}\label{cor:extra}
Suppose $(r,g,d)\,\neq \, (3,2,2)$.
Then we have that $\Br(\lM_{\tau}(r,\Lambda))\,=\,0$.
\end{cor}

\begin{proof}
For $r=2$, this result is proved as in Theorem \ref{theorem1}.
As we know it for $d> r(2g-2)$, we assume that $d \leq r(2g-2)$.

Let $r\geq 3$. Suppose first that $(r,g,d)\neq (3,2,\text{even})$,
that is, $(r,g,d)\neq (3,2,2)$, $(3,2,4)$, $(3,2,6)$.
As $d>0$, we still have a projective bundle
$\pi_2: \cU_M(r,\Lambda)  \longrightarrow  M(r-1,\Lambda)$. Therefore
there is an exact sequence as in (\ref{e2}). Note that Proposition
\ref{prop:Um-UM} and the Purity Theorem imply that
$\Br(\lM_{\tau_M^-}(r,\Lambda))=\Br(\cU_M(r,\Lambda))$.
Now using Proposition \ref{pr-br} and
Corollary \ref{cor1} and (\ref{e2})
it follows that $\Br(\lM_{\tau_M^-}(r,\Lambda))=0$.
The result follows.

Let us deal with the missing cases $(r,g,d)=(3,2,4)$, $(3,2,6)$.
We start with the case $(r,g,d)=(3,2,4)$. Let
 $$
  Z=\{E\in M(3,\Lambda) \, | \, H^1(E)\neq 0\}.
 $$
For $E\in M(3,\Lambda)\setminus Z$, we have that $\dim H^0(E)=4+3(1-g)=1$. So
the projective bundle
 $$
 \pi_1:\cU_m(3,\Lambda)\setminus \pi_1^{-1}(Z) \longrightarrow M(3,\Lambda)\setminus Z\,
 $$
is actually an isomorphism.
In this situation, the exact sequence
 \begin{equation}\label{eqn:pextra}
  {\mathbb Z}\cdot \text{cl}({\mathcal P}_0)\,\longrightarrow\,
  \Br(M(3,\Lambda)\setminus Z)\,\longrightarrow\,
  \Br(\cU_m(3,\Lambda)\setminus \pi_1^{-1}(Z)) \,\longrightarrow\, 0\,
 \end{equation}
satisfies that $\text{cl}({\mathcal P}_0)=0$.
The proof of Proposition \ref{pr-br} works also for
$M(3,\Lambda)\setminus Z$, so $\text{cl}({\mathbb P}_x)=0$.
We shall see below that
\begin{equation}\label{eqn:codime}
  \codim Z\geq 2 \qquad \text{and} \qquad \codim \pi^{-1}_1(Z)\geq 2\, .
 \end{equation}
{}From this, $\Br(M(3,\Lambda)\setminus Z)=\Br(M(3,\Lambda))$ and
$\Br(\cU_m(3,\Lambda)\setminus \pi_1^{-1}(Z))=\Br(\cU_m(3,\Lambda))=\Br(\lM_{\tau_m^+}
(3,\Lambda))$. By Corollary \ref{cor1}, $\text{cl}({\mathbb P}_x)=0$ generates
$\Br(M(3,\Lambda))$, so $\Br(M(3,\Lambda))=0$ and $\Br(\lM_{\tau_m^+}
(3,\Lambda))=0$, as required.

To see the codimension estimates (\ref{eqn:codime}), we work as follows.
Let $E\in Z\subset M(3,\Lambda)$. So $H^1(E)\neq 0$, i.e.
$H^0(E^\ast\otimes K_X)\neq 0$. Thus there is an exact sequence
 \begin{equation}\label{eqn:ww}
 0\longrightarrow \cO \longrightarrow E'\,=\, E^\ast\otimes K_X
\longrightarrow F\longrightarrow 0\, ,
 \end{equation}
for some sheaf $F$. Note that $\deg(F)=\deg (E')=2$, and $E'$ is stable
(since $E$ stable $\implies$ $E^\ast$ stable $\implies$ $E'=E^\ast \otimes K_X$ stable).
Here $F$ must be a rank $2$ semistable sheaf, since any quotient
$F\to Q$ with $\mu(Q)<\mu(F)=1$, would satisfy that $\mu(Q)<\mu(E')=\frac23$, violating the
stability of $E'$. In particular, $F$ is a (semistable) bundle, and it is parametrized
by an irreducible variety of dimension $\dim M(2,\Lambda)=3(g-1)=3$
(recall that $\dim M(r,\Lambda)=(r^2-1)(g-1)$). Now the bundle $E'$ in
(\ref{eqn:ww}) is given by an extension in ${\mathbb P}(H^1(F^\ast))$. As $H^0(F^\ast)=0$ (by semistability),
we have that $\dim {\mathbb P}(H^1(F^\ast))= -(-2 + 2(1-g))-1=3$. So the bundles $E'$
are parametrized by a $6$-dimensional variety, and therefore $\dim Z=6$ and $\codim Z=3$.

Now let us see that $\dim \pi_1^{-1}(Z) \,\leq\, 7$. Let $E\in Z$ and
$F$
as in (\ref{eqn:ww}), and note that the determinant of $F$ is fixed.
Recalling that $\dim H^1(F^\ast)=4$, we see that
we have to check that
 $$
 \dim \cF  + 3 + \dim H^0(E) -1 \,\leq\, 7\, ,
 $$
where $\cF$ is the family of the bundles $F$.
 Now $\dim H^0(E)=\dim H^1 (E)+1=\dim H^0(E')+1\leq \dim H^0(F)+2$. Hence we
only need to show that
 \begin{equation}\label{eqn:inequality}
 \dim \cF_i + \dim H^0(F)\,\leq\, 3\, ,
 \end{equation}
for $F\in \cF_i$, where $\cF=\bigsqcup \cF_i$ is the family (suitably stratified)
of the possible bundles $F$.

We have the following possibilities:
\begin{enumerate}
\item $F=L_1\oplus L_2$, where $L_1,L_2$ are line bundles of degree one,
$L_2=\det(F)\otimes L_1^{-1}$.
The generic such $F$ moves in a $2$-dimensional family, and $H^0(F) =0$.
If $\dim H^0(F)\neq 0$, then it should be either $L_1=\cO(p)$ or
$L_2=\cO(q)$, $p,q\,\in \, X$. In this case $F$ moves in a $1$-dimensional family,
and $\dim H^0(F)\leq 2$, so (\ref{eqn:inequality}) holds.

\item $F$ is a non-trivial extension $L\to F\to L$, where $L$ is a line bundle of
degree one. As $\det(F)=L^2$ is fixed,
then there are finitely many possible $L$. Now $\dim \Ext^1(L,L)=2$, so the bundles $F$
move in a $1$-dimensional family. Again
$\dim H^0(F)\leq 2$, so (\ref{eqn:inequality}) is
satisfied.

\item $F$ is a non-trivial extension $L_1\to F\to L_2$, where $L_1,L_2$ are non-isomorphic line
 bundles of degree one. As $\dim \Ext^1(L_2,L_1)=1$, we have that $F$ moves in $2$-dimensional
 family. If $\dim H^0(F)=1$ then (\ref{eqn:inequality}) holds. Otherwise, it must be $L_1=\cO(p)$ and
 $L_2=\cO(q)$, hence $F$ moves in a $1$-dimensional family and $\dim H^0(F)\leq 2$. So
 (\ref{eqn:inequality}) holds again.

\item $F$ a rank $2$ stable bundle and $H^0(F)=0$. This is clear, since $\dim M(2,\Lambda)=3$.

\item $F$ a rank $2$ stable bundle and $H^0(F)=1$. Then we have an exact sequence
$\cO\to F\to L$, where $L$ is a (fixed) line bundle of degree two. As $\dim H^1(L^\ast)=3$,
we have that $F$ moves in a $2$-dimensional family and (\ref{eqn:inequality}) holds.

\item $F$ a rank $2$ stable bundle, $\cO\to F\to L$, $\dim H^0(L)=1$ and $\dim H^0(F)=2$.
The connecting map $H^0(L)={\mathbb C} \to H^1(\cO)$ is given by multiplication by
the extension class in $H^1(L^\ast)$ defining $F$. To have $\dim H^0(F)=2$, this connecting
map must be zero, hence the extension class is in $\ker (H^1(L^\ast) \to H^1(\cO))$. This
kernel is one-dimensional (since the map is surjective). So the family of such $F$ is zero-dimensional,
and (\ref{eqn:inequality}) is satisfied.

\item $F$ a rank $2$ stable bundle, $\cO\to F\to L$, $\dim H^0(L)=2$ and $\dim H^0(F) \geq 2$.
Now it must be $L=K_X$. The connecting map
  $$
  c_\xi: H^0(K_X) \to H^1(\cO)=H^0(K_X)^*
  $$
is given by multiplication with the extension class $\xi$ in $H^1(L^\ast)=H^0(K_X^2)^*$ defining
$F$. So $c_\xi \in \mathrm{Hom}(H^0(K_X), H^0(K_X)^*)=H^0(K_X)^* \otimes  H^0(K_X)^*$ is
the image of $\xi$ under $H^0(K_X^2)^* \to H^0(K_X)^* \otimes  H^0(K_X)^*$. But this
map is the inclusion $H^0(K_X^2)^*=\text{Sym}^2 H^0(K_X)^* \subset \bigotimes^2 H^0(K_X)^*$.
This means that $c_\xi \in \text{Sym}^2 H^0(K_X)^*$.

If $\dim H^0(F)=2$, then $c_\xi$ is not an isomorphism.
The condition $\det(c_\xi)=0$ gives a $2$-dimensional family of $\xi\in H^1(L^\ast)$. So
the family of such bundles $F$ is one-dimensional and (\ref{eqn:inequality}) is satisfied.
If $\dim H^0(F)=3$, then $c_\xi=0$, and so $\xi=0$, which is not possible (since $F$
does not split).
\end{enumerate}

Finally, we tackle the case $(r,g,d)=(3,2,6)$. Now $\cU_m(3,\Lambda)\to M(3,\Lambda)$
is a projective fibration (with fibers $\PP^2$),
so Corollary \ref{cor1} and the exact sequence (\ref{e0})
imply that $\Br(\cU_m(3,\Lambda))=0$. To complete the proof
that $\Br(\lM_{\tau_m^+}(3,\Lambda))=0$, it only remains to show that
$\codim \cS_{\tau_m^+} \geq 2$.

Consider the family $\cF$ of strictly semistable bundles $E$
with $(E,\phi)\in \cS_{\tau_m^+}$.
We stratify $\cF=\bigsqcup \cF_j$, such that $\dim H^0(E)$
is constant on each $\cF_j$. We have to prove that
 $$
 \dim \cF_j + \dim H^0(E)-1 \leq \dim \lM_{\tau_m^+}(3,\Lambda)-2=10-2=8.
 $$
For $E$ strictly semistable, we have either an exact sequence $L\to E\to F$ or $F\to E\to L$, where
 $L\in \Jac^2 X$, and $F$ is a semistable bundle of rank $2$ and determinant
 $\Lambda'=\Lambda \otimes L^{-1}$ (which
 is of degree $4$). Both cases are similar, so we assume the first one. There are three
 possibilities:
 \begin{enumerate}
 \item Suppose that $\dim \Hom(F,L)=0$. Then $\dim \Ext^1(F,L)=2$.
 We stratify $\Jac^2 X$ depending on $\dim H^0(L)$. For $L\neq K_X$, $\dim H^0(L)=1$;
  for $L=K_X$, $\dim H^0(L)=2$. So
 for each stratum $\cF' \subset \Jac^2 X$,
 we have that $\dim \cF' + \dim H^0(L) \leq 3$. We also stratify the
 family of rank $2$ semistable bundles $F$, according to $\dim H^0(F)$. For any such
 stratum $\cF''$, we have that
  \begin{equation}\label{eqn:extra-final}
  \dim \cF'' + \dim H^0(F) -1 \leq 4\, .
  \end{equation}
 Assuming (\ref{eqn:extra-final}), and noting
 that $\dim H^0(E)\leq \dim H^0(L)+\dim H^0(F)$, we have that,
 for the corresponding stratum $\cF_0$,
  $\dim \cF_0 + \dim H^0(E)-1 \leq 3 + 4 + 2-1=8$.

 Let us prove (\ref{eqn:extra-final}).
 For $F$ stable, we have that  $\dim \cF'' + \dim H^0(F) -1 \leq \dim \lM_{\tau_m^+}(2,\Lambda') = 4$.
 For $F$ strictly semistable, there is an exact sequence $L' \to F\to \Lambda'\otimes L'^{-1}$,
 with $L'\in \Jac^2 X$. If $L'$ is generic, then $\dim H^0(L')=\dim H^0(\Lambda'\otimes L'^{-1})=1$
 and $\dim \Ext^1( \Lambda'\otimes L'^{-1},L')=1$. So $\dim \cF'' + \dim H^0(F) -1 \leq 2 +2-1=3$.
 For $L'=K_X$, $\Lambda'\otimes L'^{-1}=K_X$ or $L'^2= \Lambda'$, we have the bounds
 $\dim H^0(L')\leq 2$, 
 $\dim H^0(\Lambda'\otimes L'^{-1})\leq 2$ and $\dim \Ext^1( \Lambda'\otimes L'^{-1},L')\leq 2$,
 giving that $\dim \cF'' + \dim H^0(F) -1 \leq 1 +4-1=4$.

 \item Suppose that $\dim \Hom(F,L)=1$. Then $\dim \Ext^1(F,L)=3$.
 There is an exact sequence $\Lambda\otimes L^{-2}\to F \to L$.
 If $\dim H^0(L)= 1$ and $\dim H^0(\Lambda\otimes L^{-2})=1$ then
 $\dim H^0(E)=3$. This case is covered by Lemma \ref{lem:codim-semistable}.
 Otherwise $L= K_X$ or $\Lambda\otimes L^{-2} =K_X$, so there are
 finitely many choices for $L$. 
 Using that $\dim H^0(E)\leq 6$ and
 $\dim \Ext^1(L, \Lambda\otimes L^{-2})\leq 2$, we get that, for the corresponding stratum $\cF_1$,
 it is
 $\dim \cF_1 + \dim H^0(E)-1 \leq 1+ 2 + 6-1 =8$.

 \item Suppose that $\dim \Hom(F,L)=2$. Then
  $F=L\oplus L$ and $\dim \Ext^1(F,L)=4$. The extension is unique, because
  the group of automorphisms of such $F$ is of dimension $4$. Note also that 
  there are finitely many choices for $L$. 
So for the corresponding family $\cF_2$, we have
  $\dim \cF_2 + \dim H^0(E)-1 \leq 6-1 =5$.
\end{enumerate}

This completes the proof of the corollary.

\end{proof}

\begin{rmk}\label{rem}
Note that $\Br(\cU_M(r,\Lambda))=0$ for $(r,g,d)\neq (3,2,\text{even})$
(use Corollary \ref{cor:extra} and Proposition \ref{prop:Um-UM}).

Also, if $d> r(2g-2)$, then
$\Br(\cU_m(r,\Lambda))=0$ for $(r,g,d)\neq (2,2,\text{even})$
(use Corollary \ref{cor:extra} and Proposition \ref{prop:Um-UM}).
Actually, in the range $d > r(2g-2)$, the proof of Theorem \ref{theorem1} shows that
$\Br(\cU_m(r,\Lambda))=\Br(\cU_M(r+1,\Lambda))$, for any $(r,g,d)$.
\end{rmk}

\begin{rmk}
Our techniques for proving Theorem \ref{main1} do not cover the case $(r,g,d)=(3,2,2)$.
So this case remains open at the moment.

This is due to the following. Working as in the proof of Corollary \ref{cor:extra}, in
the case $(r,g,d)=(3,2,2)$, we could try two approaches. First, we could look at the map
$\pi_1: \cU_m(3,\Lambda) \to M(3,\Lambda)$. We see that whereas $\dim \cU_m(3,\Lambda)=
d+(r^2-r-1)(g-1)-1= 6$, it is $\dim M(3,\Lambda)=(r^2-1)(g-1)=8$. Therefore $\pi_1$
is generically an immersion, and there is not much hope to recover the Brauer
group of $\cU_m(3,\Lambda)$ out of that of $M(3,\Lambda)$.

Second, we could look at the map $\pi_2: \cU_M(3,\Lambda) \to M(2,\Lambda)$, which is a projective
fibration with fiber $\PP^3$.
The moduli space of $S$-equivalence classes of
semistable bundles  $\overline{M}(2,\Lambda)$ is isomorphic (for $g=2$, $d\equiv 0 \pmod2$)
to $\PP^3$  (see \cite{NR}). The
locus of properly semistable bundles $Z\subset \PP^3$ is a Kummer variety: $Z=\Jac^1 X/\ZZ_2$,
whose elements are of the form $E=L\oplus (L^{-1}\otimes\Lambda)$, $L\in \Jac^1 X$.
Then $\codim Z=1$, so we do not get the vanishing of the Brauer group of
$M(2,\Lambda)=\overline{M}(2,\Lambda) \setminus Z$.

We can still try to study the map $\pi_2$ over a larger open subset of $\overline{M}(2,\Lambda)$, recalling
that $\pi_2$ extends to a map $\pi_2: \lM_{\tau_M^-}(3,\Lambda) \to \overline{M}(2,\Lambda)$.
We denote $\tilde Z=\pi_2^{-1}(Z)$.
Consider a pair $(E,\phi)\in \tilde Z$. Then $\cO \to E\to F$, where $F$ is a semistable rank $2$ bundle.
So $F$ sits in an exact sequence $L\to F \to L^{-1}\otimes \Lambda$, where $L \in \Jac^1 X$.
Then $\dim \Ext^1(L^{-1}\otimes \Lambda ,L)=1$ if $L^2\not\cong  \Lambda$,
and $\dim \Ext^1(L^{-1}\otimes \Lambda ,L)=2$ if $L^2\cong \Lambda$.
The family of \emph{non-split} properly semistable bundles is then parametrized
by $P_1=Bl_{{\mathrm{Fix}}\, \tau}  (\Jac^1 X)$, the blow-up of $\Jac^1 X$ at the fixed points of the
involution $\tau: L\mapsto L^{-1}\otimes \Lambda$. There is an obvious map
$q: P_1 \to Z$. The family of \emph{split} semistable bundles is
parametrized by $P_2 \cong Z$.  Now consider the embedding $\imath:X\hra
\Jac^1 X$, given as $p\mapsto {\cO}(p)$. This produces maps
$\imath_1:X\hra P_1$ and $\imath_2:X \to P_2$. Then for any
$L \in \big( P_1 \setminus \imath_1(X)\cup \tau(\imath_1(X)) \big) \sqcup \big( P_2 \setminus \imath_2(X) \big)$,
we have that $H^0(F)=0$ and $\dim H^1(F)=4$. As a conclusion, if $F_0=L\oplus (L^{-1}\otimes
\Lambda) \in Z \setminus  q\circ \imath_1(X) \subset \overline{M}(2,\Lambda)$, then
the fiber
  $$
  \pi_2^{-1}(F_0) \cong \PP^3 \sqcup \PP^3\sqcup (\PP^1\times \PP^1),
  $$
where the first $\PP^3$ corresponds to the space of sections of $F$ for the non-trivial extension
$L\to F \to L^{-1}\otimes \Lambda$, the second $\PP^3$  corresponds to the space of sections of $F$ for the non-trivial extension
$L^{-1}\otimes  \Lambda \to F \to L$, and the $\PP^1\times \PP^1$ corresponds to
the sections of $F_0=L\oplus (L^{-1}\otimes
\Lambda)$ (taking the quotient by the automorphisms of the bundle).

This means that the map $\pi_2$ is not a projective fibration over any open subset larger
than $M(2,\Lambda)\subset \overline{M}(2,\Lambda)$, ruling out any hope of determining the
Brauer group of $\lM_{\tau_M^-}(3,\Lambda)$ without determining it for $M(2,\Lambda)$ first.
\end{rmk}


\begin{thebibliography}{20}
\bibitem{AM} M. Artin and D. Mumford, \emph{Some elementary examples of
unirational varieties which are not rational},
Proc. London Math. Soc. \textbf{25} (1972) 75--95.

\bibitem{balaji-biswas-gabber-nagaraj:2007} V. Balaji, I. Biswas, O.
Gabber, D. S. Nagaraj, \emph{Brauer obstruction for an universal vector
bundle}, C. R. Acad. Sci. Paris \textbf{345} (2007) 265--268.

\bibitem{bertram:1994} A. Bertram, \emph{Stable pairs and stable
parabolic pairs}, J. Algebraic Geom. \textbf{3} (1994) 703--724.

\bibitem{bradlow-daskalopoulos:1991}S. Bradlow and G. D. Daskalopoulos,
\emph{Moduli of stable pairs for holomorphic bundles over Riemann
surfaces}, Internat. J. Math. \textbf{2} (1991) 477--513.

\bibitem{bradlow-garciaprada:1996} S. B. Bradlow, O. Garc\'{\i}a-Prada,
\emph{Stable triples, equivariant bundles and dimensional reduction},
Math. Ann. \textbf{304} (1996) 225--252.

\bibitem{bradlow-garciaprada:2002} S. B. Bradlow, O.
Garc\'{\i}a-Prada, \emph{An application of coherent systems to a
Brill-Noether problem}, J. reine angew. Math. 551 (2002), 123--143.

\bibitem{Ga} O. Gabber, \emph{Some theorems on Azumaya algebras}, (in:
The Brauer group), pp. 129--209, Lecture
Notes in Math., 844, Springer, Berlin-New York, 1981.

\bibitem{milne:1980} J. S. Milne,  \emph{\`{E}tale cohomology},
Princeton Mathematical Series, 33. Princeton University
Press, Princeton, N.J., 1980.

\bibitem{munoz-ortega-vazquezgallo:2007} V. Mu\~{n}oz, D. Ortega and M. J.
V\'azquez-Gallo, \emph{Hodge polynomials of the moduli spaces of pairs},
Internat. J. Math. \textbf{18} (2007) 695--721.

\bibitem{munoz:2008} V. Mu\~{n}oz, \emph{Hodge polynomials of the
moduli spaces of rank 3 pairs}, Geom. Dedicata \textbf{136} (2008)
17--46.

\bibitem{munoz:2009}  V. Mu\~{n}oz, \emph{Torelli theorem for the
moduli spaces of pairs}, Math. Proc. Cambridge Philos. Soc. \textbf{146}
(2009) 675--693.

\bibitem{NR} M. S. Narasimhan and S. Ramanan, \emph{Moduli of vector bundles
on a compact Riemann surface}, Annals of Math. (2) \textbf{89} (1969) 14--51.


\bibitem{Ne} P. E. Newstead, \emph{Introduction to
moduli problems and orbit spaces}, Tata Institute of Fundamental
Research Lectures on Mathematics and Physics, 51,
Narosa Publishing House, New Delhi, 1978.

\bibitem{thaddeus:1994} M. Thaddeus, \emph{Stable pairs, linear systems
and the Verlinde formula}, Invent. Math. \textbf{117} (1994) 317--353.

\end{thebibliography}
\end{document}